%% file: iid_conductances_230307_arXiv.tex
\documentclass[reqno, 11pt, a4paper]{amsart}
\usepackage{amssymb}
\usepackage{centernot} 
\usepackage{bbm}
\usepackage{libertine} 
\usepackage[cal=euler, scr=boondoxo]{mathalfa}
\usepackage[colorlinks=true, linkcolor=blue,citecolor=blue]{hyperref}
\usepackage[text={33pc,605pt},centering]{geometry}
\usepackage{booktabs}
\usepackage{epigraph}
\usepackage{graphicx}
\usepackage{comment}
\usepackage{cite}
\newtheorem{theorem}{Theorem}[section]

\newtheorem{proposition}[theorem]{Proposition}
\newtheorem{assumption}[theorem]{Assumption}
\newtheorem{corollary}[theorem]{Corollary}

\theoremstyle{definition}

\theoremstyle{remark}

\numberwithin{equation}{section}
\allowdisplaybreaks
\input{commands}

\begin{document}
\title[From QIP to semigroup convergence]{From quenched 
invariance principle
	to semigroup convergence with 
 applications to  exclusion processes} 
\author{Alberto Chiarini}
\address{Università degli Studi di Padova}
\curraddr{}
\email{chiarini@math.unipd.it}
\thanks{}

\author{Simone Floreani}
\address{University of Oxford}
\email{simone.floreani@maths.ox.ac.uk}
\curraddr{}
\thanks{}


\author{Federico Sau}
\address{Università degli Studi di Trieste}
\curraddr{}
\email{federico.sau@units.it}
\thanks{}
\subjclass[2020]{60K35; 60K37; 60G22; 35B27; 60F17; 82C41; 82B43.}
\keywords{Quenched invariance principle; Ergodic theorem; Symmetric exclusion process; Hydrodynamic limit}
\maketitle

\begin{abstract}
	Consider a random walk on $\Z^d$ in a translation-invariant and ergodic random environment and starting from the origin. 
	In this short note,  assuming that a quenched invariance principle for the opportunely-rescaled walks holds,   we show how to derive an $L^1$-convergence of the corresponding semigroups. We then apply this result to obtain a quenched pathwise hydrodynamic limit for the simple symmetric exclusion process on $\Z^d$, $d\ge 2$, with i.i.d.\ symmetric nearest-neighbors conductances $\omega_{xy}\in [0,\infty)$ only satisfying
	$$\Q(\omega_{xy}>0)>p_c\ ,$$
	where $p_c$ is the critical  value for bond percolation.
\end{abstract}
\section{Introduction}
Random walks in random environment received a lot of attention in the recent years, and studying their scaling limits is one of the main challenges (see, e.g., the surveys \cite{biskup_recent_2011,kumagai_survey}).
Often the limit theorem is stated in the form of a \textit{quenched invariance principle} (QIP) for the random walk \textit{starting from the origin}. Namely, calling $\Q$ the law of the environment, under some natural assumptions of stationarity and ergodicity of $\Q$, one has that, for $\Q$-a.e.\ realization of the environment, the random walk $(X_t)_{t\ge 0}$ with $X_0=0$ rescales to a process $(\mathbb X_t)_{t\ge 0}$ whose law does not depend on the specific realization of the environment.

Examples that we have in mind are random walks on $\Z^d$ with generator of the form
\begin{equation}\label{eq:generator-RW-general}
A f(x):=  \frac{1}{\nu_x}
\sum_{y\in \bbZ^d}   \omega_{xy}\left(f(y)-f(x)\right)\ ,\qquad x \in \Z^d\ ,
\end{equation}
with random bond- and site-weights $\omega_{xy}$ and $\nu_x$, respectively.  In particular, if $\omega_{xy}=\omega_{yx}$, the walk is referred to as \textit{Random Conductance Model} (see, e.g., \cite{biskup_recent_2011}), with $\nu_x$ playing the role of invariant measure for the walk: for $\nu_x\equiv 1$  one obtains the so-called variable speed random walk, while for $\nu_x=\sum_{y\in \Z^d}\omega_{xy}$ the constant-speed random walk. Another relevant example is the \textit{Bouchaud Trap Model} (see, e.g., \cite{ben-arous_cerny_dynamics_2006}): given  a collection  of i.i.d.\ $\N$-valued  random variables $\alpha=(\a_x)_{x\in\Z^d}$ satisfying, for some $\beta\in (0,1)$,
\begin{equation*}
	\Q\big(\alpha_0\ge u\big)= u^{-\beta}\left(1+o(1)\right)\ ,\quad \text{as}\ u\to \infty\ ,
\end{equation*}   
  such a model  is the continuous-time random walk  with  infinitesimal generator
\begin{equation}\label{eq:generator-BTM}
A f(x):=  
\sum_{y\sim x} \alpha_x^{a-1}\alpha_y^a\left(f(y)-f(x)\right)\ ,\qquad x \in \Z^d \ .
\end{equation} 
In this formula, the summation runs over nearest-neighbor sites of $x$ in $\Z^d$, while $a\in [0,1]$ is an additional parameter tuning the asymmetry of the model.

For all these models, QIP has been established under various conditions (see, e.g., \cite{andres_invariance2013,andres_chiarini_deuschel_slowik_2018,andres_deuschel_slowik_2015,barlow_invariance_2010,BarlowCerny,cernyEJP,biskup_recent_2011,ben-arous_cerny_dynamics_2006,bella_schaeff20,BergerDeuschel2014,BiskupChenKumgaiWang2021,DeuschelNguyenSlowik18,chen-croydon-kumagai_quenched_2015}). In many cases, the limiting process is a Brownian motion with a diffusion matrix which does not depend on the realization of the environment. In other cases, e.g.,  for the Bouchaud trap model and the constant speed random walk with i.i.d.\ and heavy-tailed $\omega_{xy}^{-1}$ for $d\ge 2$,  the scaling limit is a semi-Markov process known as \textit{Fractional Kinetics Process}  \cite{BarlowCerny}.

Such scaling limits of random walks in random environment can be useful when studying hydrodynamic limits of  interacting particle systems (IPS) in random environment. Indeed, as shown in, e.g., \cite{nagy_symmetric_2002,faggionato_bulk_2007}, for  IPS in random environment which satisfy self-duality,   the quenched convergence of the empirical density fields of the IPS (in the sense of  finite-dimensional distribution convergence) can be obtained from the quenched convergence of the semigroup of the one-particle system.

In most papers on random walks in random environment, only the convergence of the walk starting from the origin is addressed. However for applications in hydrodynamic limits of IPS in random environment, one typically needs a suitable convergence of the random walk's semigroups, which is implied by the strenghtening of the QIP from the origin in the form of an \emph{arbitrary-starting-point} QIP.  
It is worth mentioning that the problem of deriving an arbitrary-starting-point QIP was posed in \cite{rhodes_stochastic_2010}, and only recently established in \cite{chen-croydon-kumagai_quenched_2015} (see also \cite{redig_symmetric_2020,floreani_HDL_2021}) for a class of random environments for which a finer analysis on Green functions and heat kernels is available.

 In this paper, we show how to obtain from a QIP from the origin a suitable form of $L^1$-convergence of the corresponding semigroups (or pseudo-semigroups if the limiting process is not Markovian), under the only additional assumption of translation-invariance and ergodicity of the law of the underlying environment.
We then apply our result to derive a quenched functional hydrodynamic limit for the simple symmetric exclusion process on $\Z^d$, $d\ge 2$, with i.i.d.\ conductances in $[0,\infty)$ only satisfying \begin{equation*}\Q(\omega_{xy}>0)>p_c\ ,
\end{equation*}
so that bonds with strictly positive conductances percolate. The hydrodynamic equation is the heat equation with a diffusivity matrix of the form $\sigma^2 \boldsymbol I$, with $\sigma>0$ depending  on $d$ and $\Q$, but not on the specific realization of the environment.
The simple symmetric exclusion process is one of the most studied IPS and consists of many (possibly infinite) symmetric nearest-neighbor random walks, subject to the exclusion rule: only one particle per site is allowed. The hydrodynamic limit of the simple symmetric exclusion process with random conductances has been widely studied (see, e.g., \cite{faggionato2022hydrodynamic,faggionato_bulk_2007, faggionato_cluster_2008, faggionato_hydrodynamic_2009, floreani_HDL_2021, redig_symmetric_2020, jara_hydrodynamic_2011, jara_quenched_2008}); however, the assumptions on the conductances that we assume here have not been considered before.

The rest of the paper is organized as follows. In Section \ref{section: general results}, we introduce the general setting and we state the first main results of the paper, Theorem \ref{th:quenched_l1_conv_compacts} and Corollary \ref{cor:quenched_l1_conv_compacts}, on the convergence of semigroups. In that same section, we list some relevant examples to which our result applies. In Section \ref{section: SEP iid conductances}, we introduce the symmetric exclusion process on the infinite cluster with i.i.d.\ unbounded conductances and we derive its hydrodynamic limit. In Section \ref{section: other applications}, we sketch some further applications of Theorem \ref{th:quenched_l1_conv_compacts} to other kinds of exclusion processes.

\section{From QIP from the origin to semigroup $L^1$-convergence}\label{section: general results}
Consider a random environment  $\omega \in \Omega$, in the following form:
\begin{equation}\label{eq: environment}
 \omega=((\omega_{xy})_{x,y \in \Z^d},(\nu_x)_{x\in \Z^d})\ ,
\end{equation}
with $\omega_{xy}\in [0,\infty)$ and $\nu_x\in (0,\infty)$, and equip $\Omega$ with a $\sigma$-field $\cF$. The environment is sampled according to a probability measure $\Q$, which we  assume  to be invariant and ergodic under translations in $\Z^d$ (Assumption \ref{ass:ergodic}), and for which a QIP from the origin holds (possibly, for a proper subset of environments, see Assumption \ref{ass:QIP}). There are several examples of random environments satisfying such assumptions; for the reader's convenience, some of them are discussed in Section \ref{section: examples} below. 

   In what follows,  for all $z\in \Z^d$,  $\tau_z: \Omega\to \Omega$  denotes the translation map given by
\begin{equation*}
\omega=((\omega_{xy})_{x,y\in \Z^d},(\nu_x)_{x\in \Z^d})\longmapsto	\tau_z\, \omega := ((\omega_{x+z,y+z})_{x,y\in \Z^d},(\nu_{x+z})_{x\in \Z^d})\ .
\end{equation*} 
\begin{assumption}[Ergodicity of the environment]\label{ass:ergodic} All $(\tau_z)_{z\in \Z^d}$ are $\cF$-measurable. Moreover,  
	$\Q$ is invariant and ergodic under translations in $\Z^d$, i.e., $\Q(A)= \Q(\tau_z(A))$ for all $A\in \cF$ and $z\in \Z^d$, and $\Q(A)\in \{0,1\}$ for all $A\in \cF$ such that $\tau_z(A)=A$ for all $z\in \Z^d$.
\end{assumption}

Given a realization of the environment $\omega$ sampled according to $\Q$ satisfying Assumption \ref{ass:ergodic}, we consider the random walk $X$, with $X=(X_t)_{t\ge 0}$, on $\Z^d$, $d\ge 1$, with infinitesimal generator  $A=A^\omega$ given in \eqref{eq:generator-RW-general}. Furthermore,  $\mathbf P^\omega_z$ and $\mathbf E^\omega_z$ denote the law and corresponding expectation, respectively, of the random walk $X$ with $X_0=z\in \Z^d$.  Moreover, $\cD([0,\infty);\R^d)$ stands for the Polish space of $\R^d$-valued \emph{c\`{a}dl\`{a}g} paths equipped with the $J_1$-Skorokhod topology (see, e.g., \cite[\S16]{billingsley_convergence_1999}).

\begin{assumption}[QIP from the origin]\label{ass:QIP}
There exist:
\begin{enumerate}
	\item a sequence $(\theta_n)_{n\in \N}\subset (0,\infty)$;
	\item a set $\Omega_0\in \cF$ with $\Q(\Omega_0)>0$;
	\item \label{it:ass:QIP3} a process $\mathbb X=(\mathbb X_t)_{t\ge 0}$ with paths in $\cD([0,\infty);\R^d)$ and translation-invariant law $(\mathbf P^{\mathbb X}_x)_{x\in \R^d}$, that is, for all $x\in \bbR^d$, $\bbX$ under $\mathbf  P^{\mathbb X}_x$ has the same law as $x+\bbX$ under $\mathbf P^{\mathbb X}_0$;
\end{enumerate}    such that, for $\Q(\,\cdot\mid \Omega_0)$-a.e.\ $\omega$,  $X^{n}=(X_{t\theta_n}/n)_{t\ge 0}$ under $\mathbf P^\omega_0$ converges in law to $\mathbb X$ under $\mathbf P^{\mathbb X}_0$ in $\cD([0,\infty);\R^d)$ as $n\to \infty$. 
\end{assumption}
While for many examples the QIP from the origin holds for $\Q$-a.e.\ realization of the environment, the role of the set $\Omega_0\in \cF$ in Assumption \ref{ass:QIP} will become clear when looking at random walks on supercritical percolation clusters, which is the example treated in Section \ref{section: SEP iid conductances} below: in that case,  $\Omega_0$ coincides with the set of environments for which the origin belongs to the infinite cluster.
\subsection{Main result}
We are now ready to present our main result. In what follows,  $\mathbf E_x^{\mathbb X}$ denotes the expectation of $\mathbb X$ when $\mathbb X_0=x\in \R^d$. Finally, recall  $X^n=(X_{t\theta_n}/n)_{t\ge 0}$ from Assumption \ref{ass:QIP}.
\begin{theorem}\label{th:quenched_l1_conv_compacts}
	Under Assumptions \ref{ass:ergodic} and \ref{ass:QIP}, we have, for $\Q$-a.e.\ $\omega$, for all  compact sets $\cA\subset \R^d$, and for all uniformly continuous  bounded functions $G:\cD([0,\infty);\R^d)\to \R$, 
	\begin{align}\label{eq:conver_L1_on_compacts_general}
 \frac{1}{n^d} \sum_{x\in \Z^d} \IND_{\cA}(x/n)\IND_{\Omega_0}(\tau_x\,\omega) \left| \mathbf E^{\omega}_x[ G(X^n   ) ] - \mathbf E_{x/n}^{\mathbb X}[G(\mathbb X)]     \right|\xrightarrow{n\to \infty}0\ .	
	\end{align}
\end{theorem}

\begin{proof} 
Fix $L>0$ and $M>0$. Let $G:\cD([0,\infty);\R^d)\to \R$	be Lipschitz continuous with Lipschitz constant ${\rm Lip}_\cD(G)\le L$, and  bounded by $M$: for all $z, w	\in\cD([0,\infty);\R^d)$, 
		\begin{equation}
	\label{eq:Lipschitz}
	|G(z)-G(w)|\le L\,\mathrm{d}_{\cD}(z,w)\ ,\qquad \|G\|_{\cD,\infty}:=\sup_{z\in \cD([0,\infty);\R^d)}\left|G(z)\right|\le M\ .	\end{equation}
	where $\rm{d}_{\cD}(\cdot,\cdot)$ denotes the metric inducing the $J_1$-Skorokhod topology on $\cD([0,\infty);\R^d)$.

Fix $K>0$.
	Define, for all $k,n \in \N$,  $x\in \Z^d$, and $\omega \in \Omega$,
	\begin{align}\label{eq:definitions-f-g}
		\begin{aligned}
		f_{n,x}(\omega)&:=\left|\mathbf E^\omega_0\left[G(X^{n}+x/n)\right]-\mathbf E^{\mathbb X}_{x/n}\left[G(\mathbb X)\right]\right|\\
		g_n(\omega)&:=\sup\left\{f_{m,y}(\omega): m\ge n,\ |y|\ge n\right\}\ .
		\end{aligned}
		\end{align}
Let $\cA_n=\cA_n^K:=  \{y\in \Z^d: |y/n|\le K\} $; hence,  $\frac{1}{n^d}\sum_{x\in \cA_n}\IND_{\Omega_0}(\tau_x\,\omega)\, f_{n,x}(\tau_x\,\omega)$ coincides with the left-hand side of \eqref{eq:conver_L1_on_compacts_general} with $\cA\subset \R^d$ satisfying $n\cA\cap \Z^d=\cA_n$.

	For all $n,k \in \N$ with $k\le n$, by estimating $f_{n,x}\le g_k$ for $x\in \cA_n\setminus \cA_k$ and $f_{n,x}\le 2\|G\|_{\cD,\infty}\le 2M$ for $x\in \cA_k$, we get
		\begin{align}\label{eq:crucial}
		&	\frac{1}{n^d}\sum_{x\in \cA_n} \IND_{\Omega_0}(\tau_x\,\omega)f_{n,x}(\tau_x\,\omega)\le \frac{1}{n^d}\sum_{x\in \cA_n}\IND_{\Omega_0}(\tau_x\, \omega)\, g_k(\tau_x\, \omega) + \frac{2}{n^d}\sum_{x\in \cA_k}M\ .
	\end{align}
The second term on the right-hand side of \eqref{eq:crucial} vanishes, for every fixed $k \in \N$, as $n\to \infty$. As for the first term, since $\IND_{\Omega_0}\, g_k:\Omega\to \R$ is bounded, Assumption \ref{ass:ergodic} and the pointwise ergodic theorem yield, $\Q$-a.s.\ and for all $k \in \N$,
\begin{equation}\label{eq:EQ}
		\limsup_{n\to \infty}\frac{1}{n^d}\sum_{x\in \cA_n}\IND_{\Omega_0}(\tau_x\, \omega)\,g_k(\tau_x\, \omega) \le \E_\Q\left[\IND_{\Omega_0}\, g_k\right]\ .	
\end{equation} 

We conclude by showing that, $\Q$-a.s., $\IND_{\Omega_0}\,g_k\to 0$  as $k\to \infty$; the dominated convergence theorem would ensure	 that the right-hand side (thus, the left-hand side) of \eqref{eq:EQ} vanishes as $k\to \infty$.    By translation-invariance of $\mathbb X$ (see item \eqref{it:ass:QIP3} in Assumption \ref{ass:QIP}), we have
\begin{align*}
	f_{n,x}(\omega)=\left|\mathbf E^\omega_0\left[G(X^n+x/n)\right]-\mathbf E^{\mathbb  X}_0\left[G(\mathbb X+x/n)\right] \right|\ .
\end{align*}  
Moreover, by the convergence in Assumption \ref{ass:QIP} and Skorokhod's representation theorem, for $\Q(\,\cdot\mid\Omega_0)$-a.e.\ $\omega$, there exists a coupling of $(X^{n,\omega}, \mathbb X)$, say $(\bar X^{n,\omega},\bar {\mathbb X})$, with law $\bar {\mathbf P}^\omega_0$ and expectation $\bar {\mathbf E}^\omega_0$, for which the convergence in $\cD([0,\infty);\R^d)$ occurs $\bar {\mathbf P}^\omega_0$-a.s.:
\begin{align}\label{eq:a.s.convergence}\bar{\mathbf P}^\omega_0\text{-a.s.}\ ,\qquad
	\mathrm{d}_{\cD}(\bar X^n,\bar {\mathbb X})\xrightarrow{n\to \infty} 0\ .
\end{align}
By combining these facts, we get, for $\Q$-a.e.\ $\omega \in \Omega$, 
\begin{align*}
\IND_{\Omega_0}(\omega)\,g_k(\omega)&:= \sup\left\{f_{n,y}(\omega): n\ge k, |y|\ge k\right\}\\
&\le\IND_{\Omega_0}(\omega)\, \sup\left\{\left|\bar{\mathbf E}^\omega_0\left[G(\bar X^n+y/n)-G(\bar{\mathbb X}+y/n)\right]\right|: n\ge k, |y|\ge k \right\}\\
&\le \IND_{\Omega_0}(\omega)\, \sup\left\{\bar{\mathbf E}^\omega_0\left[\left(L\,\mathrm{d}_{\cD}(\bar X^n,\bar {\mathbb X})\right)\wedge 2M\right]: n\ge k\right\}\xrightarrow{k\to \infty}0\ ,
\end{align*}
where the last inequality is a consequence of \eqref{eq:Lipschitz}, while the last step follows by \eqref{eq:a.s.convergence} and the dominated convergence theorem. 

We thus have proved that, given $K,L, M >0$, there exists $\Omega^{K,L,M}\in \cF$ such that:
\begin{itemize} 
	\item $\Q(\Omega^{K,L,M})=1$;
	\item  \eqref{eq:conver_L1_on_compacts_general} holds for all $\omega \in \Omega^{K,L,M}$, for all compact sets $\cA\subset \R^d$ contained in the centered Euclidean ball of radius $K>0$, and for all Lipschitz continuous  bounded functions $G$ satisfying ${\rm Lip}_\cD(G)\le L$ and $\|G\|_{\cD,\infty}\le M$.
\end{itemize}  By taking  $\Omega':=\cap_{K,L,M\in \N}\,\Omega^{K,L,M}$, $\Q(\Omega')=1$, and \eqref{eq:conver_L1_on_compacts_general} holds for all $\omega \in  \Omega'$, for all compact sets $\cA\subset \R^d$, and for all Lipschitz continuous bounded functions on $\cD([0,\infty),\R^d)$. Since this latter function space is dense in the space of uniformly continuous  bounded functions with respect to the uniform norm $\|\cdot\|_{\cD,\infty}$, this concludes the proof.  
\end{proof}
As a direct consequence of Theorem \ref{th:quenched_l1_conv_compacts},  we obtain the following result, which will be used in the applications of Sections \ref{section: SEP iid conductances} and \ref{section: other applications}. In what follows, $\cC_c^+(\R^d)$ denotes the space of non-negative and compactly supported continuous functions on $\R^d$.
\begin{corollary}\label{cor:quenched_l1_conv_compacts} Under Assumptions \ref{ass:ergodic} and \ref{ass:QIP}, and further assuming that $\mathbb X$ has continuous sample paths, it holds that, for $\Q$-a.e.\ $\omega$ and for all $t>0$,  compact sets $\cA\subset \R^d$, and uniformly continuous  bounded functions $g:\R^d\to \R$, 
	\begin{align}\label{eq:conver_L1_on_compacts}
		\frac{1}{n^d} \sum_{x\in \Z^d} \IND_{\cA}(x/n)\IND_{\Omega_0}(\tau_x\,\omega) \left| \mathbf E^{\omega}_x[ g(X_{t\theta_n}/n) ] - \mathbf E_{x/n}^{\mathbb X}[g(\mathbb X(t))]     \right|\xrightarrow{n\to \infty}0\ .	
	\end{align}

Additionally, assume that for some  $f\in \cC_c^+(\R^d)$ the following two conditions hold true:
	\begin{align}\label{eq:conver_diff_on_compacts}
	\frac{1}{n^d} \sum_{x\in \Z^d} \IND_{\Omega_0}(\tau_x\,\omega) \left( \mathbf E^{\omega}_x[ f(X_{t\theta_n}/n) ] - \mathbf E_{x/n}^{\mathbb X}[f(\mathbb X(t))]     \right)\xrightarrow{n\to \infty}0\ ,
\end{align}
and
\begin{align}\label{eq:integrability}
\lim_{k\to \infty}\limsup_{n\to \infty}\frac{1}{n^d}\sum_{\substack{x\in \Z^d\\
	|x|>kn}} \IND_{\Omega_0}(\tau_x\,\omega)\mathbf E_{x/n}^{\mathbb X}\left[f(\mathbb X_t)\right]\, \dd x=0	\ .
\end{align}
Then, 
		\begin{align}\label{eq:conver_L1_for_compactly_supported-ftcs}
		\frac{1}{n^d} \sum_{x\in \Z^d} \IND_{\Omega_0}(\tau_x\,\omega) \left| \mathbf E^{\omega}_x[ f(X_{t\theta_n}/n) ] - \mathbf E_{x/n}^{\mathbb X}[f(\mathbb X(t))]     \right|\xrightarrow{n\to \infty}0\ .
	\end{align}
\end{corollary}
\begin{proof}
	The first claim, namely \eqref{eq:conver_L1_on_compacts}, is an immediate consequence of Theorem \ref{th:quenched_l1_conv_compacts} and the continuity of the paths of the limiting process $\bbX$, ensuring continuity of the one-time projections.
The second claim follows by combining \eqref{eq:conver_L1_on_compacts} with the arguments in the proof of \cite[Proposition 5.3]{redig_symmetric_2020}, which we briefly recall here for the reader's convenience.

Fix $f\in \cC_c^+(\R^d)$. Using $|a|=a+2\max\{-a,0\}$, $a\in \R$, one splits, for all $k>0$,  the expression in \eqref{eq:conver_L1_for_compactly_supported-ftcs} as	follows:
\begin{align*}
	&			\frac{1}{n^d} \sum_{x\in \Z^d} \IND_{\Omega_0}(\tau_x\,\omega) \left( \mathbf E^{\omega}_x[ f(X_{t\theta_n}/n) ] - \mathbf E_{x/n}^{\mathbb X}[f(\mathbb X(t))]     \right)\\
	&\quad +\frac{2}{n^d} \sum_{\substack{x\in \Z^d\\
			|x|\le kn}} \IND_{\Omega_0}(\tau_x\,\omega)\max\left\{\mathbf E_{x/n}^{\mathbb X}[f(\mathbb X(t))]   -\mathbf E^{\omega}_x[ f(X_{t\theta_n}/n) ]\,,\,0	\right\}\\
	&\quad +\frac{2}{n^d} \sum_{\substack{x\in \Z^d\\
			|x|> kn}} \IND_{\Omega_0}(\tau_x\,\omega)\max\left\{\mathbf E_{x/n}^{\mathbb X}[f(\mathbb X(t))]   -\mathbf E^{\omega}_x[ f(X_{t\theta_n}/n) ]\,,\,0	\right\}\ .
\end{align*}
The first term vanishes as $n\to \infty$ by the assumption \eqref{eq:conver_diff_on_compacts}; the second one, for every $k>0$, as $n\to \infty$ by \eqref{eq:conver_L1_on_compacts};  the third one is bounded by twice the expression in \eqref{eq:integrability}, which vanishes taking first $n\to \infty$ and then $k\to \infty$. 
\end{proof}
\subsection{Examples}\label{section: examples}
We list here six examples for which both Theorem \ref{th:quenched_l1_conv_compacts} and Corollary \ref{cor:quenched_l1_conv_compacts} apply,  the limiting process $\mathbb X$ having continuous sample paths. We divide these examples into two sub-classes, depending on the type of space-time scaling involved. In either case, we always consider environments as in \eqref{eq: environment} and satisfying  Assumption \ref{ass:ergodic}, with $\cF$ being the product Borel $\sigma$-field.

\subsubsection{Diffusive scaling}
Assumption \ref{ass:QIP} holds for $\theta_n=n^2$, $\Omega_0=\Omega$, and $\mathbb X$  a $d$-dimensional Brownian motion with a non-degenerate $\Q$-a.s.\ constant covariance matrix, if the random environment $\omega$ fulfills either one of the following conditions:
\begin{itemize}
	\item[(1)] \emph{moment conditions, $d\ge 1$:} $\omega_{xy}$ symmetric,   nearest-neighbor,   satisfying:
	\begin{align*}
		&\E_\Q\big[\omega_{xy}\big]<\infty\quad\text{and}\quad
		\E_\Q\big[\omega_{xy}^{-1}\big]<\infty\ ,\qquad d=1,2\ ,\\
		&\E_\Q\big[\omega_{xy}^p\big]
		<\infty\quad\text{and}\quad
		\E_\Q\big[\omega_{xy}^{-q}\big]<\infty\ ,\qquad d\ge 3\ \ ,
	\end{align*}	
for $p, q \in (1,\infty]$,  $\frac{1}{p}+\frac{1}{q}<\frac{2}{d-1}$, with either  $\nu_x = 1$ or $\nu_x=\sum_y\omega_{xy}$ for all $x\in \bbZ^d$; see \cite{biskup_recent_2011} for $d=1,2$, \cite{bella_schaeff20} for $d\ge 3$;
	\item[(2)] \emph{elliptic \& i.i.d., $d\ge2$:}   $\omega_{xy}$ symmetric, nearest-neighbor, \textit{i.i.d.}, satisfying
	\begin{equation*}\Q(\omega_{xy}\ge c)=1\ ,\quad\text{for some}\ c >0\ ,	
	\end{equation*}
	and  $\nu_x = 1$ for all $x\in \bbZ^d$; see \cite{barlow_invariance_2010};
	\item[(3)] \emph{long-range, $d\ge 2$:} $\omega_{xy}$ symmetric,  satisfying
	\begin{equation*}
		\E_\Q\bigg[\bigg(\sum_{x\in \Z^d}\omega_{0x}|x|^2\bigg)^p\bigg]<\infty\quad\text{and}\quad \E_\Q\bigg[\big(1/\omega_{0x}\big)^q\bigg]<\infty\ ,\ |x|=1\ ,
	\end{equation*}
for $p,q \in (1,\infty)$, $\frac{1}{p}+\frac{1}{q}<\frac{2}{d}$; moreover, $\nu_x=\sum_y \omega_{xy}$ for all $x\in \bbZ^d$ and $\E_\Q[\omega_{00}]<\infty$; see \cite{BiskupChenKumgaiWang2021}; 
	\item[(4)] \emph{balanced \& i.i.d., $d\ge 1$:} $\tilde \omega_x(y):=(\tau_x\,\omega)_{0y}$, satisfying, for every  $y\in \Z^d$ with $|y|=1$,  $\tilde \omega_z(y)=\tilde \omega_z(-y)>0$, and being i.i.d.; further, $\nu_x=\sum_y\tilde \omega_x(y)$; see \cite{BergerDeuschel2014}.
\end{itemize}

\subsubsection{Sub-diffusive scaling}\label{sec:examples-sub-diffusive}Assumption \ref{ass:QIP} holds for 
\begin{equation}
	\label{eq:theta-n}\theta_n=n^{d/\beta}\  \text{if}\ d\ge 3\ ,\qquad \text{or}\quad \theta_n= n^{d/\beta}\log^{1-1/\beta}(n)
  \  \text{if}\ d=2\ ,
  \end{equation} $\Omega_0=\Omega$,	 and $\mathbb X:=$  a multiple of the Fractional Kinetics process if $\omega$ fulfills either one of the following conditions, for some $\beta \in (0,1)$ and $c_1,c_2 >0$ (see \cite{BarlowCerny} for $d\ge 3$,  \cite{cernyEJP} for $d=2$):
\begin{itemize}
	\item[(5)] \emph{heavy-tailed conductances:} $\omega_{xy}$ symmetric, nearest-neighbor, satisfying, 
		\begin{equation}\Q(\omega_{xy}>u)=c_1u^{-\beta}(1 + o(1))\ \text{as}\ u\to \infty\ ,	\quad\text{and}\quad \Q(\omega_{xy}\ge c_2)=1\ ;
			\end{equation}
		here, $\nu_x=\sum_y\omega_{xy}$;

	\item[(6)] \emph{heavy-tailed site-weights:}  $\omega_{xy}:=\alpha_x^{a-1}\alpha_y^a$, with $a\in[0,1]$ and $\alpha_x$ i.i.d.\ satisfying
		\begin{equation}\label{eq:assumption Q BTM}
			\Q(\alpha_x>u)=c_1u^{-\beta}(1 + o(1))\ \text{as}\ u\to \infty\ ,	\quad\text{and}\quad \Q(\alpha_x\ge c_2)=1\ ;
	\end{equation} 
here, $\nu_x = 1$ for all $x\in \bbZ^d$.
\end{itemize}

The list above is far to be complete,  and we refer the interested reader also to, e.g., \cite{andres_deuschel_slowik_2015, DeuschelNguyenSlowik18, Deuschel_Guo_Ramirez_2018}, the variable-speed random walk discussed in Section \ref{section: SEP iid conductances} (cf.\  \eqref{eq:generator-RCM}), as well as to examples of time-dependent environments, e.g., \cite{andres_chiarini_deuschel_slowik_2018}.

\section{Applications to exclusion processes}\subsection{SSEP with iid conductances}\label{section: SEP iid conductances}
Consider $\Z^d$, $d\ge2$,  and let $E_d$ denote the set of unoriented nearest-neighbor bonds of $\Z^d$.  Let $\omega=(\omega_e)_{e\in E_d}$ be i.i.d.\ random variables on $[0,\infty)$, for which we assume \begin{equation}\label{eq:percolation-threshold}
	\Q(\omega_e>0)>p_c\ ,
	\end{equation} 
where  $p_c=p_c(d)\in (0,1)$ denotes the critical probability for i.i.d.\ bond percolation on $\Z^d$. 	
Define $\cO(\omega):=\{e \in E_d : \omega_e >0\}$, and, for $\Q$-a.e.\ $\omega$, set $\cC(\omega)\subset \Z^d$ to be the unique infinite connected supercritical percolation cluster of sites with at least one adjacent bond in $\cO(\omega)$. Finally,  define $\Omega_0:=\{\omega\in \Omega\ :\ 0\in \cC(\omega)\}$; then $\Q(\Omega_0)>0$.

For any $\omega$ sampled according to $\Q$,  we consider $(\eta_t)_{t\ge 0}$ as the simple symmetric exclusion process on $\cC(\omega)$ with conductances $(\omega_e)_{e\in E_d}$, i.e., the Markov process with state space $\Xi:=\{0,1\}^{\cC(\omega)}$ and generator given, for all local functions $\varphi:\Xi\to \R$, by 
\begin{align}\label{equation generator SSEP iid cond}
\cL^\omega \varphi(\eta):=\sum_{\substack{xy\in E_d\\
		x,y\in \cC(\omega)
}} \omega_{xy}\left(\varphi(\eta^{xy})-\varphi(\eta)\right)\ .
\end{align}
Here, $\eta^{xy}\in \Xi$ denotes the configuration obtained from $\eta$ by exchanging $\eta(x)$ with $\eta(y)$.  Since the single-particle system is $\Q$-a.s.\  conservative \cite[Lemma 2.11]{barlow_invariance_2010},  also the particle system $\eta_t$ is $\Q$-a.s.\ 	well-defined (see, e.g., \cite[Appendix A]{Chiarini_Flo_redig_sau}). Further, $\P_\eta$ denotes the law of $\eta_t$ when $\eta_0=\eta$, while, for a probability distribution $\mu$ on $\{0,1\}^{\cC(\omega)}$,  we write $\P_\mu:=\mu(\dd \eta)\P_\eta$.

It is well-known (see, e.g., \cite[Proposition 3.14]{chen-croydon-kumagai_quenched_2015}) that $m_n:=\frac{1}{n^d}\sum_{x\in \cC(\omega)}\delta_{x/n}$  converges vaguely in a $\Q$-a.e.\ sense to a deterministic multiple of the Lebesgue measure of $\R^d$. More precisely, setting $q=q(\Q,d):=\Q(0\in \cC(\omega)) \in (0,1)$, $\Q$-a.s., the following convergence
\begin{align}\label{eq:conv-measures-cluster}
	\int_{\R^d} f\, \dd m_n \underset{n\to \infty}\longrightarrow q\int_{\R^d} f\, \dd x\ ,\qquad f \in \cC^+_c(\R^d)\ ,
\end{align}
holds.
Our goal is to determine the scaling limit of the diffusively rescaled empirical density fields of the particle system on $\cC(\omega)$, i.e., 
\begin{equation}
\cX^n_t\ :=\ \frac{1}{n^d}\sum_{x \in \cC(\omega)}\eta_{tn^{2}}(x)\, \delta_{x/n}\ ,\qquad t\ge 0\ .
\end{equation}
In what follows, we view $\cX^n_t$ as a random measure in $\cM_v(\R^d)$, i.e., the space of locally-finite measures endowed with the vague topology. Moreover, we write $\cD([0,\infty);\cM_v(\R^d))$ for the space of  $\cM_v(\R^d)$-valued \emph{c\`{a}dl\`{a}g} paths endowed with the $J_1$-Skorokhod topology.
\begin{proposition}\label{th:HL-ssep}
	Let $d\ge2$ and $\omega=(\omega_e)_{e\in E_d}$ be i.i.d.\ and fulfilling \eqref{eq:percolation-threshold}. Assume that, $\Q$-a.s., a sequence of probability distributions $(\mu_n)_n$ on $\{0,1\}^{\cC(\omega)}$ satisfies
	\begin{equation}\label{eq: assumption initial condition}
\lim_{n\to \infty}\P_{\mu_n}\left(\left|\int_{\R^d}f\,\dd \cX^n_0 - q\int_{\R^d} f\, \gamma\,\dd x\right|>\varepsilon\right)=0\ ,\qquad f\in \cC_c^+(\R^d)\ ,\ \varepsilon >0\ ,
	\end{equation}
for some deterministic and continuous $\gamma :\R^d\to [0,1]$. Then, there exists $\sigma=\sigma(\Q,d)>0$ such that, $\Q$-a.s., the family $((\cX^n_t)_{t\ge 0})_n$ is  tight in $\cD([0,\infty);\cM_v(\R^d))$ and satisfies
	\begin{equation}\label{eq:hl}
	\lim_{n\to \infty}\P_{\mu_n}\left(\left|\int_{\R^d}f\,\dd \cX^n_t - q\int_{\R^d} f\, \rho_t\,\dd x\right|>\varepsilon\right)=0\ ,\qquad t >0\ ,\ f\in \cC_c^+(\R^d)\  ,\ \varepsilon >0\ ,
\end{equation}
  where $(\rho_t)_{t\ge 0}$ is the unique strong solution, on $\R^d$, to
\begin{equation}\label{eq: heat eq}
	\partial_t \rho_t= \sigma^2 \Delta \rho_t\ ,\quad \text{with}\ \rho_0=\gamma\ .
\end{equation}
\end{proposition}
We emphasize that it suffices to assume that the  i.i.d.\ conductances percolate (in particular,  without any moment assumptions) so to ensure that  the quenched hydrodynamic limit is non-degenerate. This result  comes as a direct consequence of Corollary \ref{cor:quenched_l1_conv_compacts}, self-duality of the symmetric exclusion process, and the strategy developed in \cite{nagy_symmetric_2002}  and further refined in, e.g., \cite{faggionato_bulk_2007,faggionato2022hydrodynamic,redig_symmetric_2020,Chiarini_Flo_redig_sau}.
We highlight the main steps below.

\smallskip

Fix a realization of the environment $\omega$ sampled according to $\Q$,  and consider the variable-speed random walk $X=(X_t)_{t\ge 0}$ on $\cC(\omega)\subset\Z^d$, having
\begin{equation}\label{eq:generator-RCM}
A^\omega f(x)=\sum_{y\in \cC(\omega)}\omega_{xy}\left(f(y)-f(x)\right)\ ,\qquad x \in \cC(\omega)\ .
\end{equation}
as its infinitesimal generator.
Recall that $\mathbf P^{\omega}_x$ and $\mathbf E^{\omega}_x$ denote the law and corresponding expectation of $X_t$ when $X_0=x\in \cC(\omega)$. Moreover, let, for all $n \in \N$ and continuous bounded functions $g\in \cC_b(\R^d)$,
\begin{equation}\label{eq:Pnt}
	P^n_t g(x/n):= 	\mathbf E^{\omega}_x\left[g(X_{tn^2}/n)\right]\ ,\qquad x \in \cC(\omega)\ ,\ t \ge 0\ ,	
\end{equation}
denote the semigroup of the diffusively rescaled random walk $X^n=(X_{n^2t}/n)_{t\ge 0}$.

We now recall the scaling limit of the  random conductance model with i.i.d.\ unbounded conductances,  first obtained in \cite{barlow_invariance_2010} under the assumption that $\Q(\omega_{xy}\ge 1)=1$, and
 further generalized in \cite{andres_invariance2013}. Here,  $B^\sigma=(B^{\sigma}_t)_{t\ge 0}$ denotes a $d$-dimensional Brownian motion with diffusion matrix $\sigma^2 \boldsymbol I$, while  $( S^{\sigma}_t)_{t\ge 0}$ its semigroup.
\begin{theorem}[{\cite[Theorem 1.1]{andres_invariance2013}}]\label{theorem: QIP RCM}
Let $d\ge2$ and $\omega=(\omega_e)_{e\in E_d}$ be i.i.d.\ and fulfilling \eqref{eq:percolation-threshold}. Then, for  $\Q(\,\cdot\mid\Omega_0)$-a.e.\ $\omega$, $X^n$  under $\P^{\omega}_0$ converges in law to $B^{\sigma}$ with $B^\sigma_0=0$.
\end{theorem}
As a consequence of the above result, Assumptions \ref{ass:ergodic} and \ref{ass:QIP}, and, thus, \eqref{eq:conver_L1_on_compacts}, hold true in this case.
Furthermore,
\begin{itemize}
	\item  due to symmetry of $P^n_t$ with respect to the counting measure on $\cC(\omega)$, \eqref{eq:conv-measures-cluster}  and the symmetry of $S^\sigma_t$  with respect to the Lebesgue measure on $\R^d$, condition \eqref{eq:conver_diff_on_compacts} holds true in this case;
	\item $S^\sigma_t: \cC_c^+(\R^d)\to \cC_b(\R^d)\cap L^1(\R^d)$; thus, condition \eqref{eq:integrability} holds true in this case.
\end{itemize} Hence,  \eqref{eq:conver_L1_for_compactly_supported-ftcs} in Corollary \ref{cor:quenched_l1_conv_compacts} holds true and thus, owing to  $\IND_{\Omega_0}(\tau_x\,\omega) = \IND_{x\in \cC(\omega)}$, for $\Q$-a.e.\ $\omega$, $t\ge 0$, and $f\in \cC_c^+(\R^d)$, 
	\begin{align}\label{eq: L1 convergence RCM}
	\lim_{n\to \infty}   \frac{1}{n^d} \sum_{x\in \cC(\omega)}  \left| P^n_t f(x/n) - S^\sigma_tf(x/n) \right|=0\ .	
\end{align}

The claim in \eqref{eq: L1 convergence RCM} is the main step in the proof of Proposition \ref{th:HL-ssep}.
Given \eqref{eq: L1 convergence RCM}, 	the proof of convergence of the finite-dimensional distributions for $\cX^n_t$ goes as in, e.g., \cite{nagy_symmetric_2002,faggionato_bulk_2007,floreani_HDL_2021} (see also \cite{Chiarini_Flo_redig_sau}, which bypasses the explicit graphical construction of the particle system). We just remark that for this step, one crucially exploits the fact that, for any fixed $\omega$, the symmetric exclusion process $\eta_t$  with generator given in \eqref{equation generator SSEP iid cond} and $X_t$ are in stochastic duality relation with duality function given by  $D(x,\eta)=\eta(x)$. Tightness is ensured either by the arguments in \cite[Section 5.1]{faggionato_hydrodynamic_2009} (see also \cite[Section 8]{faggionato2022hydrodynamic} for a refinement), or those in  \cite[Section 4]{Chiarini_Flo_redig_sau}. 

For the reader's convenience, we sketch the main steps of the proof of Proposition \ref{th:HL-ssep} in the following section; expert readers may jump directly to Section \ref{section: other applications}.
\subsubsection{Proof of Proposition \ref{th:HL-ssep} }\label{section: proof HDL}
By $\cX^n_0 \le m_n=\frac{1}{n^d}\sum_{x\in \cC(\omega)}\delta_{x/n}$ and \eqref{eq: L1 convergence RCM},  we get, $\Q$-a.s., 
\begin{align}\label{eq:term 2}
	\lim_{n\to \infty}	\E_{\mu_n}\left[ \left|   \int_{\R^d}  P^n_tf\, \dd\cX^n_0  -\int_{\R^d}  S^{\sigma}_{t}f\, \dd \cX^n_0 \right|\right]=0.
\end{align}
By duality and  negative dependence of ${\rm SSEP}(\omega)$, the following variance estimate holds (for a proof, see, e.g., \cite[Proposition 3.1]{Chiarini_Flo_redig_sau}): $\Q$-a.s., for all $t \ge 0$, $n \in \N$, and $f \in \cC^+_c(\R^d)$, 
\begin{align}\label{eq:negative-dependence}
\E_{\mu_n}\left[\left(\int_{\R^d} f\, \dd \cX^n_t-\int_{\R^d} P^n_t f\, \dd \cX_0^n \right)^2 \right]
\le \frac{1}{n^d}\, \int_{\R^d} f^2\, \dd m_n\xrightarrow{n\to \infty} 0\ ,
\end{align}
where the last step follows by \eqref{eq:conv-measures-cluster}, ensuring $\lim_{n\to \infty}\int f^2\, \dd m_n<\infty$.
Finally,  \eqref{eq: assumption initial condition}, $\cX^n_0\le m_n$, $S^\sigma_t f\in \cC_b(\R^d)\cap L^1(\R^d)$, $S^\sigma_t f\ge 0$,	and $\eqref{eq:conv-measures-cluster}$ yield, $\Q$-a.s., 
\begin{equation}
	\lim_{n\to \infty}\P_{\mu_n}\left(\left|\int_{\R^d} S^\sigma_t f\, \dd \cX^n_0 - q\int_{\R^d} S^\sigma_tf\, \gamma\,	 \dd x\right|>\varepsilon\right)=0\ ,\qquad \varepsilon> 0\ .
\end{equation}
The desired claim in \eqref{eq:hl} follows by the triangle inequality and  $\int S^\sigma_t f\, \gamma\, \dd x= \int f\, \rho_t\, \dd x$.

For tightness, we follow the same steps of the tightness proof in \cite{Chiarini_Flo_redig_sau}. Namely, let us fix  $f \in \cC_c^+(\R^d)$ and $\varepsilon>0$. We  need to show that for all $n\in \N$, there exists a $\Q$-measurable \emph{non-decreasing} function $\psi_n=\psi_{n,\varepsilon}:[0,T]\to [0,\infty)$ such that, $\Q$-a.s., for all $t\ge 0$, \begin{equation}\P_{\mu_n}\left(\left|\int_{\R^d} f\, \dd\cX^n_{t+h}-\int_{\R^d} f\, \dd \cX^n_t\right|>\varepsilon\,\bigg|\, \cF^n_t\right)\le \psi_n(h)\ ,\qquad   h\ge 0\ ,
\end{equation}
where $\cF^n_t:=\sigma\left( \cX^n_s: s\le t\right)$ and, $\Q$-a.s.,  
\begin{equation}
	\psi(h)\underset{h\to 0}\longrightarrow0\ ,\qquad\text{where}\ \	 \psi(h):=	\limsup_{n\to \infty}\psi_n(h)\ .
\end{equation}
Writing, for all $p \in [1,\infty)$ and $g \in \cC_b(\R^d)$,  $\left\|g\right\|_{p, n}:=  \left(\frac{1}{n^d}\sum_{x\in\cC(\omega)} |g(x/n)|^p\right)^{1/p}$ and following closely \cite[\S4]{Chiarini_Flo_redig_sau}, we have that $\psi_n(h)$ can be chosen as
\begin{align}\label{eq:psi-n-h}
	\psi_n(h):= C_1\sqrt{\big\|f\big\|_{2,n}^2-\big\|P^n_{h/2}f\big\|_{2,n}^2}+ 	C_2\,	\frac{1}{n^{d/2}}\big\|f\big\|_{2,n}\ ,
\end{align}
for some $C_1,C_2>0$ depending only on $f\in \cC_c^+(\R^d)$ and $\varepsilon>0$. (Note that $\psi_n(h)$ given above is non-decreasing in $h\ge 0$.)
The second term on the right-hand side of \eqref{eq:psi-n-h} vanishes $\Q$-a.s.\ as $n\to \infty$. As for  the first term,  by $\|P^n_{h/2}f\|_{2,n}^2=\frac{1}{n^d}\sum_{x\in\cC(\omega)}f(x/n)\,P^n_hf(x/n)$,  \eqref{eq: L1 convergence RCM} and \eqref{eq:conv-measures-cluster}, we have, $\Q$-a.s., 
\begin{equation}
	\lim_{n\to \infty} \big\|f\big\|_{2,n}^2-\big\|P^n_{h/2}f\big\|_{2,n}^2 = q\int_{\R^d} f\left(f-S^\sigma_h f\right)\dd x\ ,
\end{equation}  
which vanishes as $h\to 0$ by the strong continuity of $S^\sigma_h$ in $L^2(\R^d)$.

\subsection{Further applications}\label{section: other applications}
In this section we discuss other applications of Theorem \ref{th:quenched_l1_conv_compacts} and Corollary \ref{cor:quenched_l1_conv_compacts}.

First, observe that the quenched hydrodynamic limit for the symmetric exclusion process with symmetric random conductances satisfying either the conditions in Examples (1) and (3) from Section \ref{section: examples}  (the model is still self-dual in these cases, see, e.g., \cite[\S 4.1]{FloreaniJansenRedigWagner}) can be obtained by following the same proof strategy adopted in the previous section: in both case the hydrodynamic equation is the heat equation with a non-degenerate constant  diffusion matrix which does not depend on the realization of the environment.

As another application of Theorem \ref{th:quenched_l1_conv_compacts}, we sketch the hydrodynamic limit for the frequency field of the system of infinitely-many interacting ${\rm BTM}(a)$ analyzed in \cite{Chiarini_Flo_redig_sau}.

Briefly recalling Example (6) from Section \ref{sec:examples-sub-diffusive}, let $\alpha=(\alpha_x)_{x\in\Z^d}$ be $\N$-valued i.i.d.\ random variables sampled according to  $\Q$ satisfying \eqref{eq:assumption Q BTM}, this time with $c_1=c_2=1$. The particle system  $(\eta_t)_{t\ge 0}$ evolves on the  state space $\Xi:=\prod_{x\in\Z^d} \{0,1,\ldots, \alpha_x\}$, and its   Markov generator reads, when  acting on local functions $\varphi:\Xi\to \R$, as follows: 
\begin{align}\label{eq:generator-IPS}
	\cL^\alpha \varphi(\eta)= \sum_{x\in \Z^d}\eta(x)\, \alpha_x^{a-1}\sum_{y\,:\, |x-y|=1} \alpha_y^a\left(1-\eta(y)/\alpha_y\right) \left(\varphi(\eta^{x,y})-\varphi(\eta)\right)
	\ ,\qquad \eta \in \Xi\ ,
\end{align}
where $\eta^{x,y}:=\eta-\delta_x+\delta_y$.
 We consider the following rescaled empirical frequency fields:
\begin{equation}
t \in [0,\infty)\longmapsto \cZ^n_t:=	\frac{1}{n^d}\sum_{x\in \Z^d} \frac{\eta_{t\theta_n}(x)}{\alpha_x}\,	\delta_{x/n}\in \cM_v(\R^d)\ ,
\end{equation}
where $\theta_n=\theta_{n,\beta,d}$ is given in \eqref{eq:theta-n}.
As a direct consequence of  \eqref{eq:conver_L1_on_compacts} in Corollary \ref{cor:quenched_l1_conv_compacts}, and the main steps in the proof in \cite[Section 5]{Chiarini_Flo_redig_sau} (in particular, Remark 5.1 therein),  we obtain the following result, which is the quenched version of \cite[Proposition 2.6]{Chiarini_Flo_redig_sau} for when $d\ge 2$:
\begin{proposition}[{\cite[Proposition 2.6]{Chiarini_Flo_redig_sau}}]\label{pr:HL-IBTM}
	Let $d\ge2$. Given a uniformly continuous function $\gamma:\R^d\to [0,1]$, set, for all $n \in \N$,
	 $\mu_n=\otimes_{x\in \Z^d}\, {\rm Bin}(\alpha_x,\gamma(x/n))$ as the initial distribution  of the particle system.	Then, there exists $\sigma=\sigma(\Q,d)>0$ such that, $\Q$-a.s., 		
	\begin{align*}
	\lim_{n\to \infty}	\P_{\mu_n}\left( \left|  \int_{\R^d}  f\, \dd\cZ^n_t - \int_{\R^d}f\,\rho_t\,\dd x \right|>\varepsilon\right)=0\ ,\qquad t\ge 0\ ,\ f\in \cC_c^+(\R^d)\ ,\ \varepsilon>0\ ,
	\end{align*}
holds true, 
		where $(\rho_t)_{t\ge 0}$ is the unique bounded strong solution, on $\R^d$, to
	\begin{equation}\label{eq: FK equation}
		\frac{\partial^\beta}{\partial t^\beta}\rho_t = \sigma^{2/\beta}\Delta \rho_t\ ,\quad \text{with}\ \rho_0=\gamma\ .
	\end{equation}
	In this formula, $\frac{\partial^\beta}{\partial t^\beta}$ stands for the Caputo derivative of order $\beta \in (0,1)$, i.e., for all $t\ge 0$ and $h \in \cC^1(\R)$, 
	$\frac{\partial^\beta}{\partial t^\beta} h(t):=\frac{1}{\Gamma(1-\beta)}\int_0^t \frac{1}{(t-s)^\beta}\,h'(s)\, 	\dd s$.
\end{proposition}

	\subsection*{Acknowledgments}
	S.F.\ acknowledges financial support from the Engineering and Physical Sciences Research Council of the United Kingdom through the EPSRC Early Career Fellowship EP/V027824/1.
	  A.C., S.F.\ and F.S.\ thank the  Hausdorff Institute for Mathematics (Bonn) for its hospitality during the Junior Trimester Program \textit{Stochastic modelling in life sciences} funded by the Deutsche Forschungsgemeinschaft (DFG, German Research Foundation) under Germany’s Excellence Strategy - EXC-2047/1 - 390685813. While this work was written, A.C.~was associated to INdAM (Istituto Nazionale di Alta Matematica ``Francesco Severi'') and GNAMPA.	Finally,  S.F.\ thanks Noam Berger and Martin Slowik for useful and inspiring discussions.

\end{document}

%% file: commands.tex
\def\N{{\mathbb N}}
\def\Z{{\mathbb Z}}
\def\Q{{\mathbb Q}}
\def\R{{\mathbb R}}

\newcommand{\dd}{\text{\normalfont d}}

\renewcommand{\a}{\alpha}

\newcommand{\E}{{\mathbb E}}
\renewcommand{\P}{{\mathbb P}}

\newcommand{\IND}{\mathbbm{1}}

\newcommand{\cA}{\ensuremath{\mathcal A}}

\newcommand{\cC}{\ensuremath{\mathcal C}}
\newcommand{\cD}{\ensuremath{\mathcal D}}

\newcommand{\cF}{\ensuremath{\mathcal F}}

\newcommand{\cL}{\ensuremath{\mathcal L}}
\newcommand{\cM}{\ensuremath{\mathcal M}}

\newcommand{\cO}{\ensuremath{\mathcal O}}

\newcommand{\cX}{\ensuremath{\mathcal X}}

\newcommand{\cZ}{\ensuremath{\mathcal Z}}
\newcommand{\bbR}{\ensuremath{\mathbb R}}

\newcommand{\bbX}{\ensuremath{\mathbb X}}

\newcommand{\bbZ}{\ensuremath{\mathbb Z}}